\documentclass[12pt]{article}

\usepackage{latexsym,amssymb, authblk}
\usepackage{xcolor}
\usepackage{amsmath}
\usepackage{amsthm}
\usepackage{amsfonts}
\usepackage{ dsfont }
\usepackage{hyperref}
\usepackage{cleveref}
\usepackage{geometry}
\usepackage{mathtools}

\usepackage{tikz}
\usepackage{tikz-cd}
\usetikzlibrary{decorations.pathreplacing}
\usetikzlibrary{snakes}

\allowdisplaybreaks

\newtheorem{theorem}{Theorem}
\numberwithin{theorem}{section}

\newtheorem{corollary}[theorem]{Corollary}

\newtheorem{remark}[theorem]{Remark}
\newtheorem{example}[theorem]{Example}
\newtheorem{assumption}{Assumption} 
\newtheorem{assumptioni}{General assumptions}

\def\rmd{\mathrm{d}}
\def\Law{\mathrm{Law}}
\def\1{\mathds{1}}
\def\d{\mathbf{dist}}
\def\R{\mathbb{R}}
\def\X{\mathbb{X}}

\def \cP{\mathcal{P}}

\title{Conditions for uniform in time convergence: applications to averaging, numerical discretisations and mean-field systems}
\author[1]{Katharina Schuh}
\author[2]{Iain Souttar}
\affil[1]{Institute of Analysis and Scientific Computing, TU Wien, Wiedner Hauptstraße 8–10, 1040 Wien, \textit{katharina.schuh@tuwien.ac.at}}
\affil[2]{Centre for Applications of Mathematical \& Computing Sciences, University of Warwick, Coventry, UK, \textit{iain.souttar@warwick.ac.uk}}

\date{}

\begin{document}
\maketitle

\begin{abstract}
    We establish general conditions under which there exists uniform in time convergence between a stochastic process and its approximated system. 
    These standardised conditions consist of a local in time estimate between the original and the approximated process as well as of a contraction property for one of the processes and a uniform control for the other one.
    Specifically, the results we present provide global in time error bounds for multiscale methods and numerical discretisations as well as uniform in time propagation of chaos bounds for mean-field particle systems. We provide a general method of proof which can be applied to many types of approximation.
    
    In all three scenarios, examples where the joint conditions are verified and uniform in time convergence is achieved are given. \\
    \textbf{MSC2020 subject classifications:} Primary 60J60; secondary 60H35, 65C30, 82C31 
\end{abstract}

\section{Introduction}

There are a great variety of situations in the sciences in which two systems must be compared. In particular, one may be interested in a system which is difficult to analyse. Alongside this, there may be a system which is more tractable but does not produce precisely the same behaviour. In this case we may consider the latter as a type of \emph{approximation} to the former. If one wishes to make quantitative-- or even qualitative-- statements about the original system, then one must first understand the distance (in some metric) between the two. In this paper we study the distance between a variety of systems and some corresponding approximations, in particular providing verifiable conditions sufficient to obtain a uniform in time bound on the distance between the two. We start with two processes $(X_t)_{t \ge 0}$ and $(X_t^{\delta})_{t\ge 0}$ on some state space $\X$. 
The law of the process $(X_s^{(\delta)})_{s\ge 0}$ at time $t>0$ with initial condition $X_0^{(\delta)}\sim \nu$ is denoted by $$\nu p_t^{(\delta)} (\rmd z) = \int_{\X} \nu (\rmd x) p_t^{(\delta)} (x, \rmd z),$$ where $p_t^{(\delta)}$ is the corresponding transition function. 
Consider a distance $\d(\cdot , \cdot)$ between probability measures $\nu, \eta$ on $\X$.

In this case, one often arrives at a bound such as the following
\begin{equation}\label{eqn:introFTH}
     \sup_{t\le T}\{\d (\nu p_t, \nu p_t^{\delta})\} \le C(T) \delta^{\alpha}\, ,
\end{equation}
where, in particular, $C(T)$ grows with $T$ and $\delta$ is some small parameter, for example a discretisation parameter. This is fine for many applications, but can be troublesome when, for example, further asymptotic analysis or long-time predictions are required.

We say that there is uniform in time convergence (in $\delta$), with rate $\alpha>0$, if the following is true for all $t \geq 0$, $\delta>0$ and all probability measures $\nu$ on $\X$
\begin{equation}\label{eqn:introuit}
     \sup_{t\ge 0}\d (\nu p_t, \nu p_t^{\delta}) \le C \delta^{\alpha} \, . 
\end{equation}

 We consider uniform in time convergence in three settings, and this paper is set out as follows. In Section \ref{sec:multiscalemethods}, we prove a bound for the method of averaging for Stochastic Differential Equations (SDEs), seen in Theorem~\ref{thm:avg_meas}. In Section~\ref{sec:numdiscr} we consider the numerical discretisation of SDEs and prove a uniform in time result in Theorem \ref{thm:uit_num}. Finally, in Section~\ref{sec:meanfield} we consider mean-field particle systems described by a system of SDEs where the uniform in time result is given in Theorem~\ref{thm:mf_meas}. Each section also contains discussion on applications, with examples given of systems for which our assumptions can be verified. We note that in these three settings the vernacular is different in the corresponding literature. Results such as \eqref{eqn:introuit} are referred to as uniform in time bounds, uniform in time convergence, uniform in time propagation of chaos, or even global-in-time stability.

Proving uniform in time convergence of approximations is often a delicate task. In particular, techniques using Gr\"onwall's inequality in the standard way lead to $C(T)$ in \eqref{eqn:introFTH} growing exponentially in $T$. For many processes and corresponding approximations, a bound such as \eqref{eqn:introuit} is simply not true. There are approximations, though, for which there is substantial evidence pointing towards uniform in time convergence holding. Attention has therefore been focused on closing this theoretical gap; there have been numerous results obtaining uniform in time convergence in particular cases, and we collect them in Section \ref{sec:litrev}. However, these are isolated examples, and we are not aware of any common rubric or systematic method of proof for uniform in time convergence of approximations in the existing literature. In this paper we provide a set of conditions (see \textbf{General assumptions 1}) and a general framework (see the proof of Theorem \ref{thm:intro}), which together allow one to obtain uniform in time convergence results across a broad range of approximations.
In particular, our approach is not restricted to a specific metric but is rather formulated for a general one. To the best of our knowledge, this universal presentation is new. 
Moreover, when the uniform in time result can be shown in $L^2$ Wasserstein distance, strong error convergence can often be deduced for the same framework (see \Cref{cor:avg_strongerror} and \Cref{cor:disc_strongerror}). This being said, there are certainly uniform in time convergence results for each of our settings with specific metrics and, again, we refer to these in Section \ref{sec:litrev}.
The contribution of this paper, in our understanding, is covering all three settings under one method of proof, offering a common procedure allowing one to prove uniform in time results for more general approximations. In particular, for many approximations finite time convergence (\textbf{General assumptions 1} (2)) is well known, and the method of proof presented in Theorem \ref{thm:intro} allows one to leverage existing results to obtain uniform in time convergence. We now set out the method of proof we use throughout this paper in a general form.

Let there be two processes $(X_t)_{t \ge 0}$ and $(X_t^{\delta})_{t\ge 0}$ as above. We assume the following.
\begin{assumptioni} \label{ass:intro} We impose
\begin{enumerate}
    \item \label{ass:intro1} (Contractivity of process $(X_t)_{t\ge 0}$): There exists $\lambda>0$ such that for all probability measures $\nu$ and $\eta$ and $t>0$ it holds
    \begin{align*}
        \d (\nu p_t, \mu p_t) \le e^{-\lambda t} \d (\nu, \eta)\,.
    \end{align*}
    \item \label{ass:intro2}  (Finite time convergence) There exists a time $\tau >0$, a constant $\alpha>0$ and a function $\mathbf{M}:\mathcal{P}(\X)\to[0,\infty)$ 
    such that for all probability measures $\nu\in \mathcal{P}(\X)$ any $\delta>0$ it holds  
    \begin{align*}
        \sup_{t\le \tau} \d (\nu p_t, \nu p_t^{\delta})\le \delta^{\alpha} \mathbf{M}(\nu)\,.
    \end{align*}
    \item \label{ass:intro3} (Uniform control) For all probability measures $\nu$ on  $\X$ there exists a constant $\mathbf{C}(\nu)$, independent of $t$ and $\delta$, such that 
    \begin{align*}
        \sup_{t\ge 0} \mathbf{M}(\nu p_t^{\delta}) \le \mathbf{C}(\nu)\,.
    \end{align*}
\end{enumerate}
    
\end{assumptioni}
We note that the function $\mathbf{M}$ is independent of $\delta$. 
Further, if the $\d$ is the $L^p$ Wasserstein distance, i.e., $\d(\cdot, \cdot)=\mathcal{W}_p(\cdot, \cdot)$, then $\mathbf{M}(\cdot)$ is often connected to the $p$-th moment, e.g., $\mathbf{M}(\nu)=\Big(\int_{\X} |x|^p \nu(\rmd x)\Big)^{1/p}$.

Before presenting Theorem \ref{thm:intro}, which is representative of the types of results we obtain in this paper, we now very briefly describe the \textbf{General assumptions 1}, each in turn. The contractivity of one of the processes enforces that the initial data of the process become less important as time grows. Under certain choices of $\d$, for example the $L^p$ Wasserstein distance, this implies that there exists a unique invariant measure and the convergence to it is exponential. In particular this is the case for $\d$ being the $L^p$ Wasserstein distance on the space of probability measures with finite $p$-th moment $\mathcal{P}^p(\X)$ by Banach fixed point theorem provided $\X$ is complete and separable, cf. \cite[Corollary 33]{ollivier2010}, \cite{Villani2009}. 
The finite time convergence means that the processes are close, at least for some time, and for this amount of time one has convergence of the laws in $\delta$. We note that the time interval over which there is convergence must be independent of $\delta$. The uniform control ensures that at least one of the processes spends enough time in a central compact set and does not escape to infinity with time. Notice that both \textbf{General assumptions 1} (1) and \textbf{General assumptions 1} (3) refer to only one of the processes, and that they are different from each other, i.e. one cannot just have uniform control and contractivity of the same process.
We now present a theorem sketch of the general method of proof used throughout this paper.
\begin{theorem}[Informal theorem of uniform in time convergence]\label{thm:intro} Suppose the \textbf{General assumptions 1} hold. Let $\nu$ be a probability measure on $\X$. Then there exists a constant $\tilde{\mathbf{C}}>0$, independent of $\delta$ and $t$, such that for all $t\ge 0$,
\begin{align} \label{eq:intro_uit}
    \d (\nu p_t, \nu p_t^{\delta}) \le \tilde{\mathbf{C}}\delta^{\alpha}.
\end{align}
\end{theorem}
\begin{proof}[(Sketch of proof)]
    In the following we assume $t = k\tau$ for some $k \in \mathbb{N}$, though this is purely for ease of exposition as it can be relaxed, see the proof of Theorem \ref{thm:avg_meas}. We write
   \begin{equation*}
       \begin{split}
           \d (\nu p_t, \nu p_t^{\delta}) &\le \sum_{i=0}^{k-1}\d (\nu p_{i\tau }^{\delta} p_{(k-i)\tau}, \nu p_{(i+1)\tau}^{\delta} p_{(k-i-1)\tau}) \\
           &\le   \sum_{i=0}^{k-1}e^{-\lambda (k-i-1)\tau} \d (\nu p_{i\tau }^{\delta}p_{\tau}, \nu p_{(i+1)\tau}^{\delta} ) \\
           &\le \sum_{i=0}^{k-1}e^{-\lambda (k-i-1)\tau} \delta^\alpha \mathbf{M}(\nu p_{i\tau}^{\delta}) \\
           &\le \sum_{i=0}^{k-1}e^{-\lambda (k-i-1)\tau} \delta^\alpha \mathbf{C}(\nu) \\
           &\le   \frac{\delta^\alpha \mathbf{C}(\nu)}{1-e^{-\lambda \tau}}\,,
       \end{split}
   \end{equation*} 
   where we applied the triangle inequality and have used \textbf{General assumptions}~\ref{ass:intro} (\ref{ass:intro1}), \textbf{General assumptions}~\ref{ass:intro} (\ref{ass:intro2}) and \textbf{General assumptions}~\ref{ass:intro} (\ref{ass:intro3}) for the second, third and fourth inequalities respectively. This is independent of $k$ and so the proof is complete. See also Figure~\ref{figure1} for an illustration of the influence of the assumption to the distance of two neighbouring paths.
\end{proof}
The above method of proof is completely independent on the labelling of the processes. That is, the same proof can be repeated but with $X_t^\delta$ as $X_t$ and vice versa. In this case the assumptions are also switched. This, in theory, allows flexibility in the case that the assumptions are more easily verifiable one way around than the other. However, in some specific settings the order does matter, see Remark \ref{rem:relabelling}. Implicitly in the above we have used an assumption of some form of time homogeneity, purely by the existence of the transition functions (and by them being independent of the starting time). This is very important, and is made throughout this paper. Without this, even for very simple systems, uniform in time convergence is not guaranteed. We present such a case in Appendix \ref{sec:appTIH}. We also remark that in \textbf{General assumptions}~\ref{ass:intro} (\ref{ass:intro2}), the term $\delta^{\alpha}$ in the finite time convergence can be replaced by a general function $g:\mathbb{R}_+\to \mathbb{R}_+$ satisfying $g(\delta)\to 0$ as $\delta$ tends to zero which changes the bound in \eqref{eq:intro_uit} to $\tilde{\mathbf{C}}g(\delta)$. Moreover, the requirement of $\d$ being a distance can be relaxed to a pseudometric, since $\d(\nu,\mu)=0$ implies $\nu=\mu$ is not used in the proof.

 One particular motivation for a uniform in time bound such as \eqref{eq:intro_uit} is so that one can commute the limits of time $t$ and of the approximation parameter $\delta$. That is, the invariant measure of the limiting process coincides with the limit of the invariant measure. An elementary way of viewing this is through the Moore-Osgood theorem from real analysis, which gives sufficient conditions for two limits to commute. This is relevant to the results in this paper and we discuss it in Remark \ref{rem:MO}.

\begin{figure}[ht]
\begin{tikzpicture}
\tikzstyle{every node}=[font=\tiny]

\draw[dashed, gray!50] (0,0) -- (0, 6);
\draw[dashed, gray!50] (2,0) -- (2, 6);
\draw[dashed, gray!50] (4,0) -- (4, 6);
\draw[dashed, gray!50] (6,0) -- (6, 6);
\draw[dashed, gray!50] (8,0) -- (8, 6);
\draw[dashed, gray!50] (10,0) -- (10, 6);
\draw[dashed, gray!50] (12,0) -- (12, 6);
\draw[dashed, gray!50] (14,0) -- (14, 6);

\draw[decorate,decoration={snake,amplitude=0.5pt}, very thick, ->, blue] (0,5) -- (14,5);
\coordinate[label=above: {$\nu$} ] (A) at (0,5);
\coordinate[label=right: {$\nu p_{t}$} ] (B) at (14,5);
\coordinate[label=right: {$\nu p_{t}^{\delta}$} ] (C) at (14,2);
\fill (A) circle (2pt);
\fill (B) circle (2pt);
\fill (C) circle (2pt);

\coordinate[label=above: {$i \tau$} ] (D) at (6,1);
\coordinate[label=above: {$j\tau$} ] (E) at (8,1);
\fill (D) circle (2pt);
\fill (E) circle (2pt);
\coordinate[label=above: {$0$} ] (X) at (0,1);
\coordinate[label=above: {$t=k\tau$} ] (Y) at (14,1);
\fill (X) circle (2pt);
\fill (Y) circle (2pt);

\draw[decorate,decoration={snake,amplitude=0.5pt},->, red!50, very thick] (A) to[bend right=10] node[above] { } (C);
\draw[decorate,decoration={snake,amplitude=0.5pt},->, red!50, very thick] (6,5) to[bend right=10] node[above] { } (14,3);
\draw[decorate,decoration={snake,amplitude=0.5pt},->, red!50, very thick] (8,5) to[bend right=10] node[above] { } (14,3.5);

\draw [decorate,decoration={brace,amplitude=5pt,mirror}]
  (8,5) -- (8,4.2) node[midway,xshift=-1.3em, black!70]{$O(\delta^{\alpha})$};
  
\draw [decorate,decoration={brace,amplitude=5pt,mirror}]
  (14,3) -- (14,3.5) node[midway,xshift=3.3em, black!70]{};
\coordinate[label={[black!70]right: {$e^{-\lambda(k-j)\tau} O(\delta^{\alpha})$} }] (Z) at (14.1,3.3);

\coordinate[label=above: {$\nu p_{\tau j}$} ] (F) at (8,5);
\fill (F) circle (2pt);
\coordinate[label=below: {$\nu p_{\tau i} p_{\tau(j-i)}^{\delta}$} ] (G) at (8,4.2);
\fill (G) circle (2pt);  

\coordinate[label=right: {$\nu p_{\tau j}p_{\tau (k-j)}^{\delta}$} ] (H) at (14,3.7);
\fill (14, 3.5) circle (2pt);
\coordinate[label=right: {$\nu p_{\tau i} p_{\tau(k-i)}^{\delta}$} ] (I) at (14,2.8);
\fill (14,3) circle (2pt);

\draw [decorate,decoration={brace,amplitude=5pt,mirror,raise=4ex}]
  (0,1) -- (6,1) node[midway,yshift=-3em]{uniform control};
\draw [decorate,decoration={brace,amplitude=5pt,mirror,raise=4ex}]
  (6,1) -- (8,1) node[midway,yshift=-3em]{local error};
\draw [decorate,decoration={brace,amplitude=5pt,mirror,raise=4ex}]
  (8,1) -- (14,1) node[midway,yshift=-3em]{contractivity};
\draw[->] (0,1) -- (14, 1);
\coordinate[label=below: {time} ] (Z) at (0.3,1);

\end{tikzpicture}
\caption{{\footnotesize Comparison of two neighbouring paths in the telescoping sum and the effect of the three assumptions. The distance between all neighbouring paths is sumable and of order $\mathcal{O}(\delta^{\alpha})$. Note that in $\mathcal{O}(\delta^{\alpha})$ the uniform control $\mathbf{C}(\nu)$ is hidden.}} \label{figure1}
\end{figure}
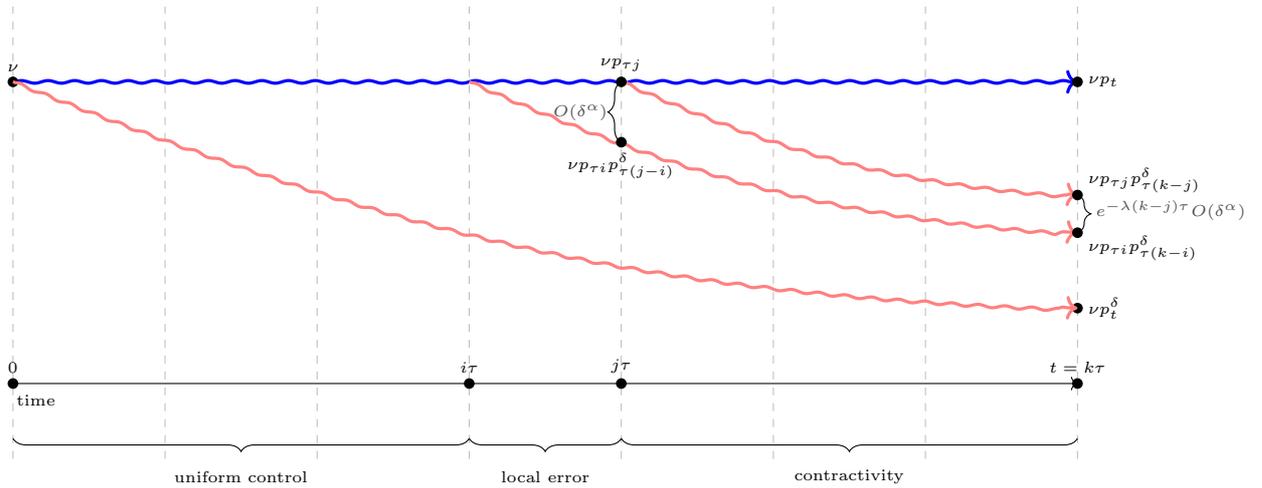

\subsection{Literature review}\label{sec:litrev}
Such is the breadth of the approximations we consider here (which, we recall, are the method of averaging for SDEs, numerical discretisations of SDEs, and mean-field limits of SDEs) and their corresponding bodies of literature, we give a non-exhaustive literature review for each. We introduce the specific notation and set-up in the respective sections, however we collect the literature review here for the ease of the reader. We emphasise again that there are existing uniform in time results (in the sense of Theorem \ref{thm:intro}) for each of the settings considered here.

There is a vast literature on the convergence of multiscale methods over finite time horizons (i.e. obtaining bounds of the form \ref{eqn:introFTH}), and we make no claim to be complete here. For a broad overview, see \cite{weinan2011principles, pavliotis2008multiscale}. See e.g. \cite{rockner2019strong,rockner2020diffusion,pardoux2001, pardoux2003} for weak error results on averaging and homogenisation on systems of SDEs in Euclidean space. Strong error results on averaging in a variety of settings can be found in \cite{rockner2019strong,hong2022strong,brehier2012strong}. The literature for uniform in time results in the case of multiscale methods (averaging for SDEs being the specific context we consider here) is smaller than in the case of numerical discretisations and propagation of chaos. For the averaging method for SDEs in Euclidean space specifically, and the results which most closely resemble those we obtain here, see \cite{crisan2024}. Specifically, Theorem \ref{thm:uit_avg_semi} given here is very similar to \cite[Theorem 3.2]{crisan2024}, albeit we present a different method of proof. There are some results for systems in the case of deterministic dynamics, see \cite{ilyin1998global} and references therein. See also \cite{flandoli, hairer2010simple} for uniform in time bounds obtained in other multiscale settings. We also flag \cite{stoltz2018longtime}, which produces a bound between the invariant measures of the coupled process and the averaged process, specifically for Langevin dynamics with a compact state space. Finally, see e.g. \cite{pavliotis2008multiscale} for the motivation of uniform in time bounds for the method of averaging.


For the setting of numerical discretisations, there are numerous works. As we mainly consider approximations of solutions to stochastic differential equations, we focus on them. An overview on the analysis of finite time weak and strong errors for numerical schemes including the Euler-Milstein scheme is provided in e.g. the textbooks \cite{milstein2004, kloeden1992}.
Considering the uniform in time error analysis, conditions on the derivative estimates are established in \cite{uitEuler,Feng_2021} to obtain specifically uniform in time weak error bounds for first order SDEs. 
There is a vast literature on studying bounds between the invariant measure of a system and its numerical approximation. For an analysis in a general framework we refer to the work by Durmus and Eberle \cite{durmus2024}, where they provide asymptotic bounds for the convergence to a given target measure (see also Section~\ref{sec:invmeas} here). This work was one of the original motivations to obtain uniform in time strong error bounds for the numerical approximation. 
For bounds between the unadjusted Langevin algorithm (ULA) to approximate first order SDEs and the invariant measure of the original SDE see \cite{dalalyan2017, durmus2017, durmus2019, Majka2020, pages2023, liu2025}.
The asymptotic analysis of kinetic samplers forming approximations of second-order SDEs are studied in \cite{leimkuhler2024, monmarche2021, sanzserna2021, schuhwhalley2024}.
Finally also for time-discrete Markov chains convergence of numerical implementable approximation is studied, as for the Hamilton Monte Carlo method \cite{duane1987,neal2011}. Both the convergence to equilibrium and bounds on the asymptotic bias for the unadjusted numerically implementable algorithm are analysed in \cite{mangoubi2018, bourabee2020, bourabee2023a, bourabee2023b, Chak2023}.

Turning to the uniform in time convergence for mean-field particle systems, this phenomena is referred to as propagation of chaos. 
This notion, which describes that under suitable conditions the particles are approximately independent, was originally introduced by Kac \cite{kac1956}, see also the work by McKean \cite{mckean1967} and the latter works by Sznitman \cite{sznitman1991} and Méléard \cite{meleard1996}. They studied the so-called McKean-Vlasov diffusion which we are particularly interested in and provided propagation of chaos bounds for finite time intervals.
In \cite{malrieu2001} uniform in time bounds are established in $L^2$ Wasserstein distance for the convex setting.  For the non-convex case \cite{durmus2020} used a reflection coupling approach to show uniform in time propagation in $L^1$ Wasserstein distance. Uniform in time propagation of chaos for Boltzmann collision processes is also studied by Mischler and Mouhot in \cite{mischler2013kac} and later by Heydecker \cite{heydecker2019pathwise} among others. 
For further results on mean-field particle systems we refer to \cite{chaintron2022a, chaintron2022b} and the literature therein.

Finally, see \cite{del2019backward} and references therein for a general result on uniform in time bounds between SDEs using a forward-backward stochastic interpolation formula.

\subsection{Convergence to invariant measure}\label{sec:invmeas}
We now make a point to place our results in the context of the literature studying long-time behaviour. As mentioned above, this type of result is well-studied in the setting of numerical discretisations (see for instance \cite{Feng_2021, leimkuhler2024}), and is obtained for the setting of averaging in \cite{stoltz2018longtime}. Formally, suppose one has the same setting as above, with a process $(X_t)_{t\ge 0}$ which has unique invariant measure $\mu_\infty$ and a process $(X_t^\delta)_{t \ge 0}$ with unique invariant measure $\mu_\infty^\delta$ (which is dependent on $\delta$). Now assume that the \emph{finite time convergence} assumption (\textbf{General assumptions 1} (\ref{ass:intro2})) holds for all $\tau$ (where $\mathbf{M}$ may now depend on $\tau$), and further that we know
\begin{equation}\label{eqn:convInvMeas}
    \d (\mu_\infty, \mu_\infty^\delta) \leq C\delta^\alpha
\end{equation}
for some $\alpha>0$. We note here that one can obtain \eqref{eqn:convInvMeas} and the modified version of \textbf{General assumptions 1} (\ref{ass:intro2}) from the three assumptions in \textbf{General assumptions 1}, see for example \cite[Remark 6.3]{matttingly2010} and \cite[Lemma 1]{durmus2024}, but this is not the focus of the argument presented here. Then a simple application of the triangle inequality gives
\begin{align}\label{eqn:triIneq}
\begin{split}
        \d(\nu p_t, \nu p_t^{\delta}) &\le  \d(\mu_\infty^\delta, \nu p_t^{\delta}) + \d(\nu p_t, \mu_\infty) + \d (\mu_\infty, \mu_\infty^\delta) \\
        &\le C\delta^\alpha + e^{-\lambda t}\, ,
\end{split}
\end{align}
which means, given any tolerance $\epsilon$ we can take $t$ large enough and $\delta$ small enough such that $\d(\nu p_t, \nu p_t^{\delta})\le \epsilon$. This asymptotic convergence is fine for many cases (in particular, it is enough to be able to commute the limits of $\delta \rightarrow 0$ and $t \rightarrow \infty$, see Remark \ref{rem:MO} for more on this), but it is not uniform in time convergence. That is, it does not give you a bound such as \eqref{eqn:introuit}. Such a bound is useful if taking certain asymptotics, as one can perform a transformation of time $t$ and retain the convergence in $\delta$.

\section{Multiscale methods}\label{sec:multiscalemethods}

\subsection{Convergence result}

Consider the following slow-fast system
	\begin{align}
	\rmd X_t^\delta &=  b(X_t^\delta, Y_t^{\delta}) \rmd t + \sigma(X_t^\delta, Y_t^{\delta}) \, \rmd W_t \label{slow} \\ 
	\rmd Y_t^\delta &= \frac{1}{\delta}g(X_t^\delta,Y_t^\delta) \rmd t + \frac{1}{\sqrt{\delta}} a(X_t^\delta, Y_t^\delta)\, \rmd B_t  \label{fast} 
	\end{align}
	with initial datum $(X_0, Y_0)=(x,y)\in \mathbb{R}^n \times \mathbb{R}^d$. Here $\delta$ is a small parameter,  $(X_t^\delta, Y_t^{\delta}) \in \mathbb{R}^n \times\mathbb{R}^d$, and $W_t$ and $B_t$, respectively, are  $n-$dimensional and a $d-$dimensional standard Brownian motion, respectively, which we take to be independent of each other.
In \eqref{slow}-\eqref{fast}, the parameter $\delta$ determines how fast the dynamics of $Y_{t}^{\delta}$ is compared to $X_{t}^{\delta}$. When $\delta \ll 1$, which is the regime in which we are interested, the process $Y^\delta_t$ evolves faster than $X^\delta_t$. For this reason, $X^\delta_t$ and $Y^\delta_t$ are commonly referred to as the slow and fast process respectively. An important object when studying averaging is the frozen process, which can informally be defined as the fast process \eqref{fast} with $\delta = 1$ and fixed slow dynamics $X_t^{\delta} = x$, though a more careful description can be found e.g. in \cite{crisan2024}. Under the condition that the frozen process associated to \eqref{fast} has a unique invariant measure $\mu_\infty (x; \rmd y)$ for each $x$, the method of averaging allows one to prove that the process $X_{t}^{\delta}$ converges as $\delta \rightarrow 0$ to the process $\bar X_t$, given by the following $\mathbb{R}^n$-valued SDE:
	\begin{equation}\label{eqn:averaged}
	\rmd \bar X_t = \bar b(\bar X_t) \rmd t +  \bar \sigma(\bar X_t) \rmd W_t    \end{equation}
	where 
	\begin{equation}\label{avgDef}
	\bar{b}(x):= \int_{\mathbb{R}^d} b(x,y) \mu_\infty (x; \rmd y), \quad  \bar{\sigma}(x):= \int_{\mathbb{R}^d} \sigma(x, y) \mu_\infty (x; \rmd y) \, .
	\end{equation}

  In the following, we consider metrics denoted by $\d$ which are independent of the part of the measures acting on $\mathbb{R}^d$.

Given a measure $\nu$ on $\mathbb{R}^n\times \mathbb{R}^d$, we denote by $\nu^x(x)=\int_{\mathbb{R}^d} \nu(x,y) \rmd y$ the marginal distribution on $\mathbb{R}^n$.

\begin{assumption}\label{ass_avg_meas} We impose
\begin{enumerate}
    \item \label{ass_avg_meas_Con} (Contractivity for the limit process) There exists $\lambda>0$ and $C<\infty$ such that for all probability measures $\nu, \eta$ on $\mathbb{R}^n$ and $t>0$ it holds
    \begin{align*}
        \d(\nu \bar{p}_t, \eta \bar{p}_t)\le C e^{-\lambda t} \d(\nu, \eta)\,.
    \end{align*}
    \item \label{ass_avg_meas_FTC}  (Finite time convergence) There exists a finite time $\tau>0$, a constant $\alpha>0$ and a function $\mathbf{M}:\mathcal{P}(\mathbb{R}^n\times \mathbb{R}^d)\to[0,\infty)$ such that for all probability measures $\nu\in \mathcal{P}(\mathbb{R}^n\times \mathbb{R}^d)$ and any $\delta>0$ it holds
    \begin{align*}
        \sup_{t\in[0,\tau]} \d ( \nu^x\bar{p}_{t}, (\nu p_{t}^{\delta})^x)\le \delta^{\alpha} \mathbf{M}(\nu)\,.
    \end{align*}
    \item \label{ass_avg_meas_UMB} (Uniform bound for the coupled process) For all probability measures $\nu \in\mathcal{P}(\mathbb{R}^n\times \mathbb{R}^d)$, there exists a constant $\mathbf{C}(\nu)< \infty$ such that for all $\delta>0$
    \begin{align*}
        \sup_{t\ge 0} \mathbf{M}(\nu p_t^{\delta}) \le \mathbf{C}(\nu)\,.
    \end{align*}
\end{enumerate}    
\end{assumption}
\begin{theorem}\label{thm:avg_meas}
    Suppose \Cref{ass_avg_meas} holds. Let $\nu$ be a probability measure on $\mathbb{R}^n\times \mathbb{R}^d$. Then there exists a constant $\tilde{\mathbf{C}}>0$ such that for all $t\ge 0$ and $\delta>0$,
    \begin{align}
        \d (\nu^x\bar{p}_t, (\nu p_t^{\delta})^x)\le \tilde{\mathbf{C}}\delta^{\alpha}\,.
    \end{align}
    The constant $\tilde{\mathbf{C}}$ depends on the initial distribution $\nu$ and the constants $C$, $\lambda$ and $\tau$ given in \Cref{ass_avg_meas}.
\end{theorem}
\begin{proof}
    Let $\delta>0$ and $t=k\tau+m$ for some $k\in\mathbb{N}$ and $m\in [0,\tau)$. Let $\nu\in\mathbb{R}^n\times \mathbb{R}^d$. Then, using the telescoping sum and triangle inequality, we obtain 
    \begin{align*}
        &\d\Big( \nu^x \bar{p}_t, (\nu p_t^{\delta})^x \Big)
        \\ & \le \d\Big( \nu^x \bar{p}_{k\tau}\bar{p}_m, (\nu p_{k\tau}^{\delta})^x \bar{p}_m \Big)+\d\Big( (\nu p_{k\tau}^{\delta})^x \bar{p}_m, (\nu p_t^{\delta})^x \Big)
        \\ & \le \sum_{i=0}^{k-1} \d\Big( (\nu p_{(k-i-1)\tau}^{\delta})^x\bar{p}_{(i+1)\tau+m},(\nu p_{(k-i)\tau}^{\delta})^x \bar{p}_{i\tau+m} \Big)+\d\Big( (\nu p_{k\tau}^{\delta})^x \bar{p}_m, (\nu p_t^{\delta})^x \Big)\, . 
    \end{align*}
    By \Cref{ass_avg_meas}(\ref{ass_avg_meas_FTC})-(\ref{ass_avg_meas_UMB}), the last term is bounded by 
    \begin{align*}
        \d\Big( (\nu p_{k\tau}^{\delta})^x \bar{p}_m, (\nu p_t^{\delta})^x \Big)\le \delta^{\alpha} \mathbf{M}(\nu p_{k\tau}^{\delta})\le \delta^{\alpha} \mathbf{C}(\nu)\,,
    \end{align*}
    while applying iteratively \Cref{ass_avg_meas}(\ref{ass_avg_meas_Con}), \Cref{ass_avg_meas}(\ref{ass_avg_meas_FTC}) and \Cref{ass_avg_meas}(\ref{ass_avg_meas_UMB}), the sum is bounded by 
    \begin{align*}
        &\sum_{i=0}^{k-1} \d\Big( (\nu p_{(k-i-1)\tau}^{\delta})^x\bar{p}_{(i+1)\tau+m},(\nu p_{(k-i)\tau}^{\delta})^x \bar{p}_{i\tau+m} \Big)
        \\ & \le \sum_{i=0}^{k-1} C e^{-\lambda (i\tau +m)} \d\Big((\nu p_{(k-i-1)\tau}^{\delta})^x\bar{p}_{\tau} ,  (\nu p_{(k-i)\tau}^{\delta})^x \Big)
        \\ & \le \sum_{i=0}^{k-1} C e^{-\lambda (i\tau +m)} \delta^{\alpha} \mathbf{M}(\nu p_{(k-i-1)\tau}^{\delta}) \le \sum_{i=0}^{k-1} C e^{-\lambda (i\tau +m)} \delta^{\alpha} \mathbf{C}(\nu) \le C\mathbf{C}(\nu) \delta^{\alpha} \frac{e^{-\lambda m}}{1-e^{-\lambda \tau}}\, .
    \end{align*}
    Hence, we obtain
    \begin{align*}
        &\d\Big( \nu^x \bar{p}_t, (\nu p_t^{\delta})^x \Big)\le \delta^{\alpha} \mathbf{C}(\nu) \Big(1 + C\frac{1}{1-e^{-\lambda \tau}} \Big)\, .
    \end{align*}
\end{proof}

Given two processes $(\bar{X}_t)_{t\ge 0}$, $(\bar{X}_t')_{t\ge 0}$ given by \eqref{eqn:averaged} with initial values $\bar{X}_0=x$ and $\bar{X}_0'=x'$ we write $\mathbb{E}_{(x,x')}[|\bar{X}_t-\bar{X}_t'|^2]=\mathbb{E}[|\bar{X}_t-\bar{X}_t'|^2|\bar{X}_0=x, \bar{X}_0'=x']$. 
Further, we write $\mathbb{E}_{(x,y)}[|X_t^{\delta}-\bar{X}_t|^2]=\mathbb{E}[|X_t^{\delta}-\bar{X}_t|^2 |(X_0^{\delta}, Y_0^\delta)=(x,y), \bar{X}_0=x]$.
Then, modifying Assumption~\ref{ass_avg_meas} slightly, we obtain a uniform in time strong error bound.

\begin{corollary}[Uniform in time strong error] \label{cor:avg_strongerror}
    Assume that there exists constants $\lambda, C\in (0,\infty)$ such that for all initial conditions $x,x' \in \mathbb{R}^n$ and $t\ge 0$ it holds $\mathbb{E}_{(x,x')}[|\bar{X}_t-\bar{X}_t'|^2]^{1/2}\le C e^{-\lambda t}|x-x'|$. Assume that there exists a constant $\alpha>0$, a finite time $\tau >0$ and a function $\mathbf{M}:\mathbb{R}^{n}\times\mathbb{R}^d\to [0,\infty)$ independent of $\delta$ such that for all $(x,y)\in \mathbb{R}^n\times \mathbb{R}^d$ the local strong error satisfies $\sup_{t\le \tau} \mathbb{E}_{(x,y)}[|X_t^{\delta}-\bar{X}_t|^2]^{1/2}\le \delta^{\alpha} \mathbf{M}(x,y)$. 
    Further suppose that for all $(x,y)\in\mathbb{R}^n\times \mathbb{R}^d$ there exists a constant $\mathbf{C}(x,y)$ such that for the process $(X_t^{\delta},Y_t^{\delta})$ starting in $(x,y)$ it holds $\sup_{t\ge 0} \mathbb{E}_{(x,y)}[\mathbf{M}((X_t^{\delta}, Y_t^{\delta}))]\le \mathbf{C}(x,y)$. Then, for all $(x,y)\in\mathbb{R}^n\times\mathbb{R}^d$ there exists a constant $\tilde{\mathbf{C}}>0$ such that
    \begin{align*}
        \sup_{t\ge 0} \mathbb{E}_{(x,y)}[|X_t^{\delta}-\bar{X}_t|^2]^{1/2}\le \tilde{\mathbf{C}}\delta^{\alpha}\, .
    \end{align*}
\end{corollary}

\begin{proof}
    The proof works analogously to the proof of Theorem~\ref{thm:avg_meas}.
\end{proof}

\begin{remark}
    Corollary \ref{cor:avg_strongerror} gives a modified result showing strong convergence for the method of averaging. It is well known that, in general, when $\sigma$ in \eqref{slow}-\eqref{fast} is dependent on both $X_t^\delta$ and $Y_t^\delta$ (as opposed to just $X_t^\delta$), strong convergence does not hold for the method of averaging. See e.g. \cite[Chapter 2.3.2]{weinan2011principles} for some simple examples. We point out that to verify the modified Assumption \ref{ass_avg_meas} (1), one needs strong convergence over a finite time horizon, which itself would rule out $\sigma$ in \eqref{slow}-\eqref{fast} being dependent on both $X_t^\delta$ and $Y_t^\delta$, at least in general. Hence, this result is compatible with prior results, though the point that $\sigma$ in \eqref{slow}-\eqref{fast} must be independent of $Y_t^\delta$ to achieve uniform in time strong convergence is, in some sense, hidden.
\end{remark}

\begin{remark}\label{rem:relabelling}
Unlike with the more general Theorem \ref{thm:intro} presented in the introduction, in the setting of averaging we cannot simply relabel our processes and repeat the same proof, obtaining the conditions of Assumption \ref{ass_avg_meas} with the measures swapped. The reason for this is that the coupled process and hence its law are defined on different spaces to the averaged process (and its law). While the slow-fast system \eqref{slow}-\eqref{fast} is defined on $\R^n \times \R^d$, the averaged process \eqref{eqn:averaged} is defined on $\R^n$ only. This means that one cannot just compose the transition functions together when doing the telescoping sum in the proof of Theorem \ref{thm:avg_meas}. In other words, neither $\nu p_t^\delta p_t$ (for $\nu$ a probabilitiy measure defined on $\R^n \times \R^d$) nor $\nu p_t p_t^\delta$ (for $\nu$ a probabilitiy measure defined on $\R^n$)  make sense. Instead, one can marginalise out the fast variable defined on $\R^d$ to consider $ (\nu p_t^\delta)^x p_t$, for $\nu$ a probabilitiy measure defined on $\R^n \times \R^d$. This is what we do in the proof of Theorem \ref{thm:avg_meas}. On the other hand, there is no natural way of `expanding' the measure $\nu p_t$ to make sense of $\nu p_t p_t^\delta$, outside of the very specific case where one has dynamics which preserve an initial product measure structure.
\end{remark}

\begin{remark}[Commutativity of limits.]\label{rem:MO}
More generally, one can ask the following question: given a process $X_t$ and a family of processes $X^\delta_t$ indexed by $\delta>0$, where $X^\delta_t$ converges to $X_t$ in some metric as $\delta \rightarrow 0$, $X^\delta_t$ has an invariant measure $\mu^\delta$ and $X_t$ has an invariant measure $\mu$ (which in many settings are implied by Assumption~\ref{ass_avg_meas} (see also the description in Section~\ref{sec:invmeas}), then under which circumstances do the limits $\delta \rightarrow 0$ and $t \rightarrow \infty$ commute? That is, under what conditions does the following hold
\begin{equation*}
    \lim_{\delta\rightarrow 0} \lim_{t \rightarrow \infty} X^\delta_t =  \lim_{t \rightarrow \infty}\lim_{\delta\rightarrow 0}  X^\delta_t\, .
\end{equation*}
Of course, among other things, this depends on the metric that one is studying. Let us, rather, consider the convergence of the observables. That is we want to define conditions which imply
\begin{equation}\label{eqn:doublelimit}
    \lim_{\delta\rightarrow 0} \lim_{t \rightarrow \infty} \mathcal{P}^\delta_t f(x) =  \lim_{t \rightarrow \infty}\lim_{\delta\rightarrow 0}  \mathcal{P}^\delta_t f(x) = \lim_{t\rightarrow \infty, \delta \rightarrow 0} \mathcal{P}_t^\delta f(x)\, ,
\end{equation}
for some class of functions $f$, where $\mathcal{P}_t^\delta$ is the semigroup associated with $X^\delta_t$. To this end, the Moore-Osgood theorem (see for example \cite{rudin1964principles}) states that, for $x$ fixed, provided one of the convergences $ \lim_{t \rightarrow \infty} \mathcal{P}^\delta_t f(x) = \mu^\delta(f), \lim_{\delta\rightarrow 0} \mathcal{P}^\delta_t f(x) = \bar{\mathcal{P}}_t f(x)$ is uniform (in $\delta$ for the former and in $t$ for the latter), the double limit exists and we have that \eqref{eqn:doublelimit} holds. In fact, the Moore-Osgood theorem implies more than this, specifically that either both the convergences are uniform, or neither are. This means, for example, if the convergence in $\delta$ is uniform in $t$ then it is also true that the convergence in $t$ is uniform in $\delta$. While one does not obtain a rate with this reasoning, it provides a way of linking uniform in time convergence (in $\delta$) to uniform in $\delta$ convergence (in time).
\end{remark}


Additionally, we can show uniform in time convergence for the semigroups of the coupled process and the averaged process.
Here we adapt the assumptions from the discretisation case in \cite{angeli2023uniform} to the multiscale framework. We first introduce some notation.

We let $\cP^\delta_t$ be the semigroup associated with the coupled process \eqref{slow}-\eqref{fast}, that is, 	$$
	(\cP^\delta_th)(x,y):= \mathbb{E} \left[h(X_t^\delta, Y_t^\delta) \vert (X_0, Y_0)=(x,y)\right] \,.
	$$
 We let $\bar \cP_t$ be the semigroup associated with the process \eqref{eqn:averaged}, i.e.
	$$
	(\bar{\cP}_t f)(x):= \mathbb{E}\left[f(\bar{X}_t)\vert \bar X_0 = x\right],  \qquad f \in \mathcal{C}_b(\mathbb{R}^n)\,. $$

  Let $\mu$ be the invariant measure of \eqref{slow}-\eqref{fast}, which we assume to exist and be unique, and $\bar{\mu}$ be the invariant measure of \eqref{eqn:averaged}, which again we assume to exist and be unique. Now we are ready to state our assumptions.

\begin{assumption}\label{ass_avg_semi} We impose
\begin{enumerate}
    \item \label{ass_avg_semiSES} (Strong exponential stability) There exists $\lambda>0$ and $C_0>0$ such that for all $t>0$ and $f\in \mathcal{C}_b^2(\mathbb{R}^n)$ with $\bar{\mu}(f) = 0$
    \begin{align*}
        \|\bar{\cP}_t f \|_{\mathcal{C}_b^2}\le C_0 \|f\|_{\mathcal{C}_b^2}e^{-\lambda t }\, . 
    \end{align*}
    \item \label{ass_avg_semiLEB} (Local error bound) There exists a positive function $\phi:\mathbb{R}^{n}\times \mathbb{R}^d\to [0,\infty)$ and constants $C_1>0$, $\alpha>0$ and $\tau>0$ such that for all $\delta>0$ and $(x,y)\in \mathbb{R}^n\times \mathbb{R}^d$ it holds
    \begin{align}
        |\mathbb{E}_{(x,y)}[f(\bar{X}_{\tau})-f(X_{\tau}^{\delta})]| \le C_1 \|f\|_{\mathcal{C}_b^2} \phi(x,y) \delta^{\alpha}\, .
    \end{align}
    \item \label{ass_avg_semiUC} (Uniform control) There exists a positive function $\Phi:\mathbb{R}^{n}\times \mathbb{R}^d\to [0,\infty)$  such that for all $(x,y)\in \mathbb{R}^n\times \mathbb{R}^d$ it holds
    \begin{align*}
        \sup_{t\ge 0}|\mathbb{E}_{(x,y)}[\phi(X_t^{\delta},Y_{t}^{\delta})]|\le \Phi(x,y)\, .
    \end{align*}
\end{enumerate}
\end{assumption}
\begin{remark}
    The proof of Theorem \ref{thm:uit_avg_semi} does not depend upon a specific norm. One can consider a more general norm $\| \cdot\|$, replacing the $\mathcal{C}_b^2$ norms throughout Assumption \ref{ass_avg_semi} and thus on the right hand side of Theorem \ref{thm:uit_avg_semi}. For verification purposes, we restrict to $\mathcal{C}_b^2$.
\end{remark}

  \begin{remark}
The difficulty in comparing the fully coupled dynamics to the averaged process, as also described in Remark \ref{rem:relabelling}, is that the former is defined on $\mathbb{R}^n \times \mathbb{R}^d$ while the latter is defined on $\mathbb{R}^n$. The averaging principle most often operates under the circumstances in which the slow variable, \ref{slow}, is the variable which captures the interesting behaviour, and hence the value of the fast variable \eqref{fast} is simply marginalised out. In e.g. \cite{crisan2024}, bounds are proven corresponding to the semigroup, which introduces the issue that the semigroups do not act on the same space of functions. To avoid this issue, only functions $f$ which are independent of the fast variable. i.e. functions for which $f(x,y) = f(x,y') \eqqcolon f(x)$ for all $x \in \mathbb{R}^n, y,y' \in \mathbb{R}^d$. We note that this is the same as marginalising. To see this in the case where there is a density (denoted $\rho^\delta_t(x,y ; x',y')$), write
\begin{align*}
    (\cP^\delta_t f )(x,y) &= \int_{\mathbb{R}^n} \int_{\mathbb{R}^d} f(x') \rho^\delta_t(x,y;x',y') \rmd y'\rmd x' \\
    &= \int_{\mathbb{R}^n} f(x')\int_{\mathbb{R}^d}  \rho^\delta_t(x,y;x',y') \rmd y'\rmd x' \\
    &=  \int_{\mathbb{R}^n} f(x') \tilde{\rho}^\delta_t(x,y;x')\rmd x' \\
    &\eqqcolon (\tilde{\cP}^{\delta, x}_t f )(x)
\end{align*}
where $\tilde{\rho}^\delta_t(x,y;x')$ is the marginalised density $\tilde{\rho}^\delta_t(x,y;x') \coloneqq \int_{\mathbb{R}^d}  \rho^\delta_t(x,y;x',y') \rmd y'$ and $\tilde{\cP}^{\delta, y}_t $ is the semigroup associated with the marginalised process.
  \end{remark}

\begin{theorem} \label{thm:uit_avg_semi}
    Suppose \Cref{ass_avg_semi} holds. Then for all $(x,y)\in \mathbb{R}^n\times \mathbb{R}^d$, $\delta>0$ and $t\ge 0$
    \begin{align*}
        |\mathbb{E}_{(x,y)}[f(X_t^{\delta})]-\mathbb{E}_x[f(\bar{X}_t)]|\le \frac{C(\|f\|_{\mathcal{C}_b^2}+ \bar{\mu}(f))\Phi(x,y)}{1-e^{-\lambda \tau}}
    \end{align*}
    with $C=C_1 (\max(1,C_0))$.
\end{theorem}
\begin{proof} The proof follows the strategy of the proof of \cite[Proposition 2.2]{angeli2023uniform}.
We first prove the result for $f\in \mathcal{C}_b^2(\mathbb{R}^n)$ such that $\bar{\mu}(f) = 0$, before generalising.
    We introduce the notation $\bar{\cP}_{t_1} \cP_{t_2}^{\delta} f(x,y)= \mathbb{E}[f(\bar{X}_{t_2}(X_{t_1}^{\delta}(x,y)))]$ where the process $\bar{X}_{t_2}(X_{t_1}^{\delta}(x,y))$ is a sequence of the process $X_t^{\delta}$ starting in $(x,y)$ up to time $t_1$ and the process $\bar{X}_t$ starting in $X_{t_1}^{\delta}(x,y)$ and running for time $t_2$.
    Fix $\delta>0$ and let $t=\tau k$ for some $k\in \mathbb{N}$. Then, for $0\le m\le k$
    \begin{align*}
        D_k(x,y)&:=\cP_{k\tau}^{\delta}f(x,y)-\bar{\cP}_{k\tau}f(x)
        \\ &=(\cP_{k\tau}^{\delta} f(x,y)-\bar{\cP}_{(k-m)\tau} \cP_{m\tau}^{\delta} f(x,y))+(\bar{\cP}_{(k-m)\tau} \cP_{m\tau}^{\delta} f(x,y)-\bar{\cP}_{k\tau}f(x))
        \\ & =: A_{m,k}+B_{m,k}\, .
    \end{align*}
    We observe that by \Cref{ass_avg_semi},
    \begin{align*}
        |B_{1,k}|&=|\bar{\cP}_{(k-1)\tau} \cP_{\tau}^{\delta} f(x,y)-\bar{\cP}_{k\tau}f(x)|
        \\ & =|\mathbb{E}_{(x,y)}[\bar{\cP}_{(k-1)\tau} f({X}_{\tau}^{\delta})-\bar{\cP}_{(k-1)\tau}f(\bar{X}_{\tau})]|
        \\ & \le C_1 \|\bar{\cP}_{(k-1)\tau} f\|_{\mathcal{C}_b^2}\phi(x,y)\delta^{\alpha}
        \\ & \le C_1 C_0 \| f\|_{\mathcal{C}_b^2}e^{-\lambda (k-1)\tau}\phi(x,y)\delta^{\alpha}
    \end{align*}
    and
    \begin{align*}
        |A_{1,k}|&=|\cP_{k\tau}^{\delta} f(x,y)-\bar{\cP}_{(k-1)\tau} \cP_{\tau}^{\delta} f(x,y)|
        \\ & = |\mathbb{E}_{(x,y)}[\cP^{\delta}_{(k-1)\tau}f(X_{\tau}^{\delta},Y_{\tau}^{\delta})-\bar{\cP}_{(k-1)\tau}f(X_{\tau}^{\delta})]|\le\mathbb{E}_{(x,y)}[|D_{k-1}(X_{\tau}^{\delta},Y_{\tau}^{\delta})|]\,.
    \end{align*}
    Next, we prove via induction that for $k\ge 2$ it holds
    \begin{align} \label{eq:ind_hypo}
        |D_k(x,y)|=C \|f\|_{\mathcal{C}_b^2}\sum_{i=0}^{k-1}\cP_{i\tau}^{\delta} \phi(x,y)\delta^{\alpha}e^{-\lambda \tau(k-1-i)} \,.
    \end{align}
    For $k=2$ we observe 
    \begin{align*}
        |A_{1,2}|&= |\mathbb{E}_{(x,y)}[\cP^{\delta}_{\tau}f(X_{\tau}^{\delta},Y_{\tau}^{\delta})-\bar{\cP}_{\tau}f(X_{\tau}^{\delta})]
        \\ & \le C_1 \|f\|_{\mathcal{C}_b^2}\mathbb{E}_{(x,y)}[\phi(\bar{X}_{\tau},Y_{\tau}^{\delta}) \delta^{\alpha}]= C_1 \|f\|_{\mathcal{C}_b^2}\cP_{\tau}^{\delta}\phi(x,y) \delta^{\alpha}\, .
    \end{align*}
    Hence, 
    \begin{align*}
        |D_2(x,y)|&\le |A_{1,2}|+|B_{1,2}| \le C_1\|f\|_{\mathcal{C}_b^2} (\cP_{\tau}^{\delta}\phi(x,y) +C_0 e^{-\lambda (k-1)\tau}\phi(x,y))\delta^{\alpha}
        \\ & \le C \|f\|_{\mathcal{C}_b^2}\sum_{i=0}^{1}\cP_{i\tau}^{\delta} \phi(x,y)\delta^{\alpha}e^{-\lambda \tau(1-i)}
    \end{align*}
    and \eqref{eq:ind_hypo} holds for $k=2$. Next, assume that the induction hypothesis \eqref{eq:ind_hypo} holds for $k-1$. Then,
    \begin{align*}
        |D_{k}(x,y)|& \le |A_{1,k}|+|B_{1,k}| \le \mathbb{E}_{(x,y)}[|D_{k-1}(X_{\tau}^{\delta},Y_{\tau}^{\delta})|]+ C_1 C_0 \| f\|_{\mathcal{C}_b^2}e^{-\lambda (k-1)\tau}\phi(x,y)\delta^{\alpha}
        \\ & \le \mathbb{E}_{(x,y)}\Big[C \|f\|_{\mathcal{C}_b^2}\sum_{i=0}^{k-2}\cP_{i\tau}^{\delta} \phi(X_{\tau}^{\delta},Y_{\tau}^{\delta})\delta^{\alpha}e^{-\lambda \tau(k-2-i)}\Big]+ C_1 C_0 \| f\|_{\mathcal{C}_b^2}e^{-\lambda (k-1)\tau}\phi(x,y)\delta^{\alpha}
        \\ & = C \|f\|_{\mathcal{C}_b^2}\sum_{i=1}^{k-1}\cP_{i\tau}^{\delta} \phi(x,y)\delta^{\alpha}e^{-\lambda \tau(k-1-i)}+ C_1 C_0 \| f\|_{\mathcal{C}_b^2}e^{-\lambda (k-1)\tau}\phi(x,y)\delta^{\alpha}
        \\ & =  C \|f\|_{\mathcal{C}_b^2}\sum_{i=0}^{k-1}\cP_{i\tau}^{\delta} \phi(x,y)\delta^{\alpha}e^{-\lambda \tau(k-1-i)}\, .
    \end{align*}
    Finally, we observe that by the uniform control imposed in \Cref{ass_avg_semi},
    \begin{align}\label{eqn:meanzerofinal}
        |D_{k}(x,y)|\le C \|f\|_{\mathcal{C}_b^2}\sum_{i=0}^{k-1}\Phi(x,y)\delta^{\alpha}e^{-\lambda \tau(k-1-i)}\le C \|f\|_{\mathcal{C}_b^2} \frac{\Phi(x,y)\delta^{\alpha}}{1-e^{-\lambda \tau}}\,,
    \end{align}
    which concludes the proof for mean-zero $f$.
    Now take $f \in \mathcal{C}_b^2(\mathbb{R}^n)$ such that $\bar{\mu}(f) \neq 0$. By \eqref{eqn:meanzerofinal} above we have
    \begin{align*}
        \begin{split}
             |\mathbb{E}_{(x,y)}[f(X_t^{\delta})]-\mathbb{E}_x[f(\bar{X}_t)]| &\le |\mathbb{E}_{(x,y)}[f(X_t^{\delta})-\bar{\mu}(f)]-\mathbb{E}_x[f(\bar{X}_t)-\bar{\mu}(f)]|  \\
             &\leq \frac{C(\|f- \bar{\mu(f)}\|_{\mathcal{C}_b^2}+ \bar{\mu}(f))\Phi(x,y)}{1-e^{-\lambda \tau}} \\
             &\leq \frac{C(\|f\|_{\mathcal{C}_b^2}+ \bar{\mu}(f))\Phi(x,y)}{1-e^{-\lambda \tau}}\,,
        \end{split}
    \end{align*}
    which concludes the proof.
\end{proof}

\begin{remark}
    The conditions of Assumption \ref{ass_avg_semi} can be found in the literature. Conditions on the averaged coefficients implying Assumption \ref{ass_avg_semi} (\ref{ass_avg_semiSES}) can be found in e.g. \cite{angeli2023uniform}. Alternatively, conditions on the coefficients of the coupled system implying such 
 averaged semigroup derivative estimates can be found in e.g. \cite{crisan2024}. Assumption \ref{ass_avg_semi} (\ref{ass_avg_semiLEB}) is the conclusion of \cite[Theorem 2.5]{rockner2019strong}, where $\phi(x) = C(1+|x|^m)$ for some $m>0$. Assumption \ref{ass_avg_semi} (\ref{ass_avg_semiUC}) is also implied by the assumptions of \cite[Theorem 2.5]{rockner2019strong}, as pointed out in that proof.
\end{remark}

\subsection{Application}

For a particular slow-fast system we verify Assumption~\ref{ass_avg_meas} and Assumption ~\ref{ass_avg_semi}.

\begin{example}
    Consider the following example, which can also be seen in \cite{crisan2024}, though we collect it here for the purposes of verifying our assumptions.
    	\begin{align}
	\rmd X_t^\delta &=  (-X_t^\delta -r \cos (y)) \rmd t + \sqrt{2} \, \rmd W_t \label{slowEx} \\ 
	\rmd Y_t^\delta &= \frac{1}{\delta}(-Y_t^\delta+r\sin (x)) \rmd t + \frac{ \sqrt{2}}{\sqrt{\delta}} \, \rmd B_t  \label{fastEx} \,.
	\end{align}
 In this case the averaged process is

 \begin{align} \label{ex:avg}
     \rmd \bar{X}_t = -\bar{X}_t -r\sqrt{\frac{1}{e}}\cos (\sin(\bar{X}_t))\rmd t + \sqrt{2} \rmd W_t\,.
 \end{align}
 Assumption \ref{ass_avg_semi} (\ref{ass_avg_semiSES}) holds when $|r|\leq 3.5$, since for this set of values the drift in \eqref{ex:avg} is monotonic. More specifically, it satisfies Assumption $2.6$ and Assumption $2.9$ in \cite{angeli2023uniform} and we can use \cite[Lemma 2.10]{angeli2023uniform} so Assumption \ref{ass_avg_semi} (\ref{ass_avg_semiSES}) holds.
 Assumption  \ref{ass_avg_semi} (\ref{ass_avg_semiLEB}) holds by, for example, \cite[Theorem 2.5]{rockner2019strong}.
 Assumption \ref{ass_avg_semi} (\ref{ass_avg_semiUC}) holds by \cite[Lemma 6.2]{crisan2024}.
Therefore we can apply \ref{thm:uit_avg_semi}.

To verify Assumption~\ref{ass_avg_meas}, let us assume for simplicity $r\le e^{1/2}$.
Consider two synchronously coupled copies of the averaged process \eqref{ex:avg}, i.e., $(\bar{X}_t^1,\bar{X}_t^2)_{t \ge 0}$ with initial distribution $\nu$ and $\eta$ on $\mathbb{R}$, respectively. Then by Ito's formula and Lipschitz continuity of sine and cosine it holds
\begin{align*}
    \rmd |\bar{X}_t^1-\bar{X}_t^2|^2 &= 2(\bar{X}_t^1-\bar{X}_t^2)\cdot (-\bar{X}_t^1- r e^{-1/2}\cos(\sin(\bar{X}_t^1))+\bar{X}_t^2+ r e^{-1/2}\cos(\sin(\bar{X}_t^2))) \rmd t
    \\ & \le - 2 |\bar{X}_t^1-\bar{X}_t^2|^2 +  |\bar{X}_t^1-\bar{X}_t^2| r e^{-1/2} |\cos(\sin(\bar{X}_t^1))- \cos(\sin(\bar{X}_t^2))| \rmd t
    \\ & \le - 2(1-re^{-1/2})|\bar{X}_t^1-\bar{X}_t^2|^2 \rmd t\, .
\end{align*}
Taking expectation and using Grönwall's inequality, we obtain
\begin{align*}
    \mathcal{W}_2(\nu \bar{p}_t, \eta \bar{p}_t)^2\le  \mathbb{E}[|\bar{X}_t^1-\bar{X}_t^2|^2]\le e^{ - 2(1-re^{-1/2})t} \mathbb{E}[|\bar{X}_0^1-\bar{X}_0^2|^2]\, .
\end{align*}
Taking the infimum over all couplings of $\eta$ and $\nu$ and the square root, we obtain contractivity in $L^2$ Wasserstein distance with rate $\lambda=1-re^{-1/2}$. Hence Assumption~\ref{ass_avg_meas}(\ref{ass_avg_meas_Con}) holds.
Assumption~\ref{ass_avg_meas}(\ref{ass_avg_meas_FTC}) holds by \cite[Theorem 2.1]{rockner2019strong}. More precisely, consider the slow-fast process \eqref{slowEx}-\eqref{fastEx} with initial distribution $\nu$ on $\mathbb{R}^2$ and the averaged system \eqref{ex:avg}  with initial distribution $\nu^x$ on $\mathbb{R}$. Then, for some $\tau>0$, there exists $\mathbf{M}(\nu)$ such that 
\begin{align*}
    \sup_{t\in[0,\tau]} \mathcal{W}_2((\nu p_t)^x,\nu^x \bar{p}_t)^2\le\sup_{t\in[0,\tau]} \mathbb{E}[|X_t^{\delta}-\bar{X}_t|^2]\le \mathbf{M}(\nu) \delta\, .
\end{align*}
To show Assumption~\ref{ass_avg_meas}(\ref{ass_avg_meas_UMB}), we consider the process $(X_t^{\delta}, Y_t^{\delta})$ with initial distribution $\nu$ on $\mathbb{R}^2$. Then by Ito's formula and the boundedness of sine and cosine, 
\begin{align*}
    \rmd |X_t^{\delta}|^2 & = 2X_t^{\delta}\cdot (-X_t^{\delta} - r\cos(Y_t^{\delta})\rmd t +  2X_t^{\delta}\cdot\sqrt{2}\rmd W_t + 1 \rmd t
    \\ & \le -2|X_t^{\delta}|^2 \rmd t  + 2r|X_t^{\delta}|\rmd t +  2X_t^{\delta}\cdot\sqrt{2}\rmd W_t +1 \rmd t
    \\ & \le -|X_t|^{\delta}|^2 + (r^2+1)\rmd t+  2X_t^{\delta}\cdot\sqrt{2}\rmd W_t
\end{align*}
and 
\begin{align*}
    \rmd |Y_t^{\delta}|^2 & = 2Y_t^{\delta}\cdot \delta^{-1}(-Y_t^{\delta} - r\sin(X_t^{\delta})\rmd t +  2Y_t^{\delta}\cdot\sqrt{2/\delta}\rmd B_t +  \delta^{-1} \rmd t
    \\ & \le -\delta^{-1}|Y_t|^{\delta}|^2 + (r^2+1)\delta^{-1}\rmd t+  2Y_t^{\delta}\cdot\sqrt{2\delta^{-1}}\rmd W_t\, . 
\end{align*}
Taking expectation and using Grönwall's inequality yields,
\begin{align*}
    \mathbb{E}[|X_t^{\delta}|^2]\le e^{-t} \mathbb{E}[|X_0^{\delta}|^2]+ r^2+1 \quad \text{and }\quad \mathbb{E}[|Y_t^{\delta}|^2]\le e^{-t\delta^{-1}} \mathbb{E}[|Y_0^{\delta}|^2]+ r^2+1\, ,
\end{align*}
which implies Assumption~\ref{ass_avg_meas}(\ref{ass_avg_meas_UMB}). Hence uniform in time convergence holds, i.e., there exists a constant $\tilde{\mathbf{C}}\in[0,\infty)$ such that for all $t\ge 0$
\begin{align*}
    \mathcal{W}_2(\nu^x \bar{p}_t, (\nu p_t^{\delta})^x)^2 \le \mathbb{E}[|X_t^{\delta}- \bar{X}_t|^2] \le \tilde{\mathbf{C}} \delta\, .
\end{align*}

\end{example}

\section{Numerical discretisations}\label{sec:numdiscr}
\subsection{Convergence result}
Let $(p_t)_{t\ge 0}$ denote the transition function of some time-continuous process. Here, we focus on processes $(X_t)_{t\ge 0}$ on $\mathbb{R}^d$ given by a stochastic differential equation 
\begin{align}
    \rmd X_t = b(X_t) \rmd t + \Sigma(X_t) \rmd B_t\, , \qquad \text{and } \qquad X_0=x\in\mathbb{R}^d\, ,
\end{align}
where $b:\mathbb{R}^d\to \mathbb{R}^d$, $\Sigma: \mathbb{R}^d\to \mathbb{R}^{d\times d}$ and $B_t$ is a $d$-dimensional standard Brownian motion. 
Since in general we can not implement the continuous dynamics exactly in time, we consider a numerical approximation of the dynamics. For given discretisation parameter $\delta>0$, we denote by $(\mathbf{X}_k^{\delta})_{k\in \mathbb{N}}$ the time-discrete Markov chain and by $\pi_\delta$ the corresponding transition kernel. 

Let $\d(\nu, \eta)$ be a distance between two probability measures $\nu$ and $\eta$ on $\mathbb{R}^d$. 

\begin{assumption}\label{ass} We impose:
\begin{enumerate}
\item \label{ass1} (Contractivity of continuous dynamics) There exists $\lambda>0$ and $C<\infty$ such that for all probability measures $\nu, \mu$ on $\mathbb{R}^d$ and $t>0$ it holds
\begin{align*}
\d(\nu p_t, \eta p_t)\le C e^{-\lambda t} \d(\nu, \eta)\, .
\end{align*}

\item \label{ass3}  (Local error bound) There exist a constant $\alpha>0$ and a function $\mathbf{M}:\mathcal{P}(\mathbb{R}^d)\to [0,\infty)$ such that for all probability measures $\nu$ on $\mathbb{R}^d$ and for any $\delta>0$ it holds 
\begin{align*}
\sup_{l\in\mathbb{N}: l\delta\le 1} \d (\nu p_{\delta l}, \nu \pi_{\delta}^l)\le \delta^{\alpha} \mathbf{M}(\nu)\, .
\end{align*}

\item \label{ass2}  (Uniform control for approximated dynamics) For all probability measures $\nu$ on $\mathbb{R}^d$, there exists a constant $\mathbf{C}(\nu)< \infty$ such that 
\begin{align*}
\sup_{l\in \mathbb{N}} \mathbf{M}(\nu\pi_{\delta}^l) \le \mathbf{C}(\nu)\, .
\end{align*}

\end{enumerate}
\end{assumption}

A typical example for the value $\mathbf{M}(\nu)$ would be some moment bound for $\nu$, e.g., $\mathbb{E}_{x\sim \nu}[|x|^p]$ for $p\ge 1$. Then, the third condition is a uniform $p$-th moment bound for the approximated dynamics.

\begin{theorem} \label{thm:uit_num}
Suppose \Cref{ass} holds. Let $\nu$ be a probability measure on $\mathbb{R}^d$. Then, there exists $\tilde{\mathbf{C}}<\infty$ such that for all $l\in \mathbb{N}$ and $\delta>0$ with $\delta^{-1}\in \mathbb{N}$,
\begin{align*}
\d(\nu \pi_{\delta}^l, \nu p_{\delta l}) \le \tilde{\mathbf{C}}\delta^{\alpha}.
\end{align*}
Moreover, the constant $\tilde{\mathbf{C}}$ depends on $\lambda$, $C$ and $\mathbf{C}(\nu)$ given in \Cref{ass}.

\end{theorem}

\begin{proof}
Fix $l\in \mathbb{N}$. We write $l\delta =k+m$ with $k\in \mathbb{N}$ and $m\in [0,1)$.
We assume that $\delta^{-1}\in \mathbb{N}$. Note that in this case $m\delta^{-1}=l-k\delta^{-1}\in \mathbb{N}$. 
Then 
\begin{align*}
\d&(\nu \pi_{\delta}^l, \nu p_{l \delta})= \d(\nu \pi_{\delta}^{k/\delta} \pi_{\delta}^{m/\delta} , \nu p_{k}p_m)
\\ & \le \d(\nu \pi_{\delta}^{k/\delta} \pi_{\delta}^{m/\delta},\nu \pi_{\delta}^{k/\delta} p_m) +\d(\nu \pi_{\delta}^{k/\delta} p_m, \nu p_{k}p_m)
\\ & \le \sum_{i=0}^{k-1} \d(\nu \pi_{\delta}^{(k-i)/\delta} p_{i}p_m , \nu \pi_{\delta}^{(k-i-1)/\delta} p_{i+1}p_m)+ \d(\nu \pi_{\delta}^{k/\delta} \pi_{\delta}^{m/\delta},\nu \pi_{\delta}^{k/\delta} p_m) 
\\ & = \sum_{i=0}^{k-1} \d((\nu \pi_{\delta}^{(k-i)/\delta}) p_{ i+m} , (\nu \pi_{\delta}^{(k-i-1)/\delta}p_{1}) p_{i+m})+\d(\nu \pi_{\delta}^{k/\delta} \pi_{\delta}^{m/\delta},\nu \pi_{\delta}^{k/\delta} p_m) 
\\ & \le  \sum_{i=0}^{k-1} C e^{-\lambda (i+m)} \d(\nu \pi_{\delta}^{(k-i)/\delta}  , \nu \pi_{\delta}^{(k-i-1)/\delta}p_{1} )+\d(\nu \pi_{\delta}^{k/\delta} \pi_{\delta}^{m/\delta},\nu \pi_{\delta}^{k/\delta} p_m) \,,
\end{align*}
where we used the triangle inequality in the first and second step and \Cref{ass}(\ref{ass1}) in the last step. 
Applying \Cref{ass}(\ref{ass3}) and \Cref{ass}(\ref{ass2}), it holds 
\begin{align*}
& \d(\nu \pi_{\delta}^{(k-i)/\delta}  , \nu \pi_{\delta}^{(k-i-1)/\delta}p_{1} )\le  \delta^{\alpha} \mathbf{M}(\nu \pi_{\delta}^{(k-i-1)/\delta})\le  \delta^{\alpha}\mathbf{C}(\nu), \qquad \text{and}
\\ & \d(\nu \pi_{\delta}^{k/\delta} \pi_{\delta}^{m/\delta},\nu \pi_{\delta}^{k/\delta} p_m) \le  \delta^{\alpha} \mathbf{M}(\nu \pi_{\delta}^{k/\delta})\le  \delta^{\alpha}\mathbf{C}(\nu)\, .
\end{align*}
Inserting this back into the previous equation yields
\begin{align*}
\d(\nu \pi_{\delta}^l, \nu p_{l \delta}) &\le C  \sum_{i=0}^{k-1} e^{-\lambda (i+m) } \delta^{\alpha}\mathbf{C}(\nu)+\delta^{\alpha} \mathbf{C}(\nu)\le  C \delta^{\alpha}\mathbf{C}(\nu) \frac{1-e^{-\lambda  k}}{1-e^{-\lambda }}+ \delta^{\alpha} \mathbf{C}(\nu)
\\ & \le  \delta^{\alpha}\mathbf{C}(\nu) \Big(\frac{C}{1-e^{-\lambda }}+1\Big)\, ,
\end{align*}
which concludes the proof.
\end{proof}

Given two processes $(X_t)_{t \ge 0}$ and $(X_t')_{t\ge 0}$ with initial values $X_0=x$ and $X_0=x'$ we write $\mathbb{E}_{(x,x')}[|X_t-X_t'|^2]=\mathbb{E}[|X_t-X_t|^2| X_0=x, X_0'=x']$. Further, $\mathbb{E}_{x}[|X_{l\delta}-\mathbf{X}_{l}^\delta|^2 | \mathbf{X}_0^{\delta}=x, X_0=x]$.
If we modify Assumption~\ref{ass} slightly, we obtain a uniform in time strong error bound.
\begin{corollary}[Uniform in time strong error] \label{cor:disc_strongerror} 
    Assume that there exists $\lambda>0$ and $C<\infty$ such that for all initial conditions $x,x'\in\mathbb{R}^d$ and $t\ge 0$ it holds
    $\mathbb{E}_{(x,x')}[|X_t-X_t'|^2]^{1/2}\le Ce^{-\lambda t} |x-x'|$.
    Assume that there exists a constant $\alpha>0$ and a function $\mathbf{M}:\mathbb{R}^d\to[0,\infty)$ such that for all $x\in\mathbb{R}^d$ and  $l\in \mathbb{N}$ with $l \delta \le 1$ the local strong error satisfies $\mathbb{E}_x[|X_{l\delta}-\mathbf{X}_l^{\delta}|^2]^{1/2}\le\delta^{\alpha} \mathbf{M}(x)$. Further suppose that for all $x\in\mathbb{R}^d$ there exists a constant $\mathbf{C}(x)$ such that for the process $\mathbf{X}_l^{\delta}$ starting in $x$ it holds $\sup_{l\in\mathbb{N}}\mathbb{E}_x[\mathbf{M}(\mathbf{X}_l^{\delta})]\le \mathbf{C}(x)$. 
    Then, for all $x\in \mathbb{R}^d$ there exists a constant $\tilde{\mathbf{C}}<\infty$ such that
    \begin{align*}
        \sup_{k\in \mathbb{N}} \mathbb{E}_{x}[|X_{k\delta}-\mathbf{X}_k^{\delta}|^2]^{1/2}\le \tilde{\mathbf{C}}\delta^{\alpha}\, .
    \end{align*}
\end{corollary}
\begin{proof}
    The proof works analogous to the proof of \Cref{thm:uit_num} using the modified conditions.
\end{proof}

\begin{remark}
    We note that, as in Theorem \ref{thm:intro} in the introduction, the role of the exact and the approximated dynamics can be interchanged in the assumptions, i.e., we can assume contractivity for the approximated dynamics and a uniform bound for the continuous dynamics. In this case the role of $\pi_{\delta}$ and $p_{\delta}$ are switched in the proof and uniform convergence still holds. This is in contrast to the setting of the method of averaging for SDEs, discussed in Remark \ref{rem:relabelling}.
\end{remark}

\begin{remark}
    As for the multiscale methods, uniform in time convergence can be shown for the associated semigroup. For the precise assumptions and result we refer to \cite{angeli2023uniform}.
\end{remark}

\subsection{Applications}

    Consider the \textit{overdamped Langevin dynamics} $(X_t)_{t\ge 0}$ given by 
    \begin{align}
        \rmd X_t = -\nabla U(X_t)\rmd t+ \sqrt{2}\rmd B_t\, , \qquad X_0=x\, ,
    \end{align}
    where $x\in \mathbb{R}^d$, $U:\mathbb{R}^d\to \mathbb{R}$ is a differentiable potential and $(B_t)_{t\ge 0}$ is a $d$-dimensional standard Brownian motion. Under appropriate assumptions \cite[Proposition 4.2]{pavliotis2014} on the potential $U$ a solution exists and the associated transition function $p_t$ of the dynamics has a unique invariant probability measure $\mu$ on $\mathbb{R}^d$ given by the Gibbs-distribution $\mu(\rmd x)\propto \exp(-U(x))\rmd x$. 
    The explicit Euler discretisation forms a well-known and easy to implement discretisation of the dynamics. Given the discretisation parameter $\delta>0$, the associated Markov chain $(\mathbf{X}_k^{\delta})_{k\in \mathbb{N}}$ is given by 
    \begin{align*}
        \mathbf{X}_{k+1}^{\delta}=\mathbf{X}_k^{\delta}-\delta\nabla U(\mathbf{X}_k^{\delta})+\sqrt{2\delta} \Xi_k\, ,
    \end{align*}
    where $(\Xi_k)_{k\in \mathbb{N}}$ are i.i.d random variables with $\Xi_k\sim \mathcal{N}(0,I_d)$, see e.g. \cite[Chapter 9]{kloeden1992}.
    \begin{corollary}
        Suppose $U$ is strongly convex and has a Lipschitz continuous gradient. Fix $\delta>0$. Let $\nu\in \mathcal{P}^2(\mathbb{R}^d)$. 
        Denote by $p_t$ the transition function corresponding to the overdamped Langevin dynamics and by $\pi_{\delta}$ the transition kernel corresponding to the explicit Euler discretisation with parameter $\delta>0$.
        Then, there exists a constant $\tilde{\mathbf{C}}>0$ depending on the dimension $d$, the second moment of $\nu$ and the  strongly convexity constant and the Lipschitz constant of $U$ such that for all $l\in \mathbb{N}$,
        \begin{align*}
            \mathcal{W}_2(\nu \pi_{\delta}^l, \nu p_{\delta l})\le \mathbb{E}[|X_{\delta l}-\mathbf{X}_l^{\delta}|^2|\{X_0=\mathbf{X}^\delta_0\sim \nu\}]^{1/2}  \le \tilde{\mathbf{C}}\delta\, .
        \end{align*}
    \end{corollary}
    \begin{proof}
        The result is a direct consequence of \Cref{thm:uit_num} and \Cref{cor:disc_strongerror}. By strongly convexity contraction hold see e.g. \cite[Equation (6)-(7)]{bolley2012}. By \cite[Theorem 9.6.2]{kloeden1992} a local error bound of order $\alpha=1$ holds. Finally, the uniform second moment bound follows by using the strong convexity and Lipschitz continuity of the gradient.
    \end{proof}
    \begin{remark}[Uniform in time convergence for other distance functions]
        We note that relaxing the conditions on $U$ and assuming that for instance $U$ is only strongly convex outside a Euclidean ball, we can still get uniform in time convergence in $L^1$ Wasserstein distance by using the contraction result by Eberle \cite{eberle2016}. 
    \end{remark}
    \begin{remark}[Higher order discretisation schemes]
        The results also carries over to higher order discretisation schemes. See for instance \cite[Theorem 11.5.1]{kloeden1992} for local error bounds with $\alpha>1$ in $L^2$ Wasserstein distance.
    \end{remark}

\begin{example}[Langevin dynamics]
Consider the process $(X_t,V_t)_{t\ge 0}$ given by
\begin{align}
    \begin{cases}
        \rmd X_t = V_t \rmd t \\
        \rmd V_t= -\nabla U(X_t) \rmd t -\gamma V_t + \sqrt{2 \gamma} \rmd B_t \\
    \end{cases}
\end{align}
with initial condition $(X_0,V_0)=(x,v)\in \mathbb{R}^{2d}$. As in the overdamped case $U:\mathbb{R}^d\to \mathbb{R}$ is a differentiable potential and $(B_t)_{t\ge 0}$ is a $d$-dimensional standard Brownian motion. 
Here, under appropriate assumption on $U$ a solution exists and the associated transition function $p_t$ has a unique invariant probability measure $\mu$ given by the Boltzmann-Gibbs distribution $\mu(\rmd x\rmd v )\propto \exp(-U(x)-|v|^2/2) \rmd x \rmd v$. 
Analogously to the overdamped case an explicit Euler scheme can be considered which has accuracy of order $\mathcal{O}(\delta)$ and gives uniform in time error bounds of order $\mathcal{O}(\delta)$. Alternatively more sophisticated splitting schemes can be considered which lead to global higher order accuracy, e.g. the so-called OBABO scheme \cite{monmarche2021} or the UBU scheme \cite{zapatero2019, sanzserna2021}.

Focusing on the UBU scheme, assume that  $U$ is strongly convex, has a Lipschitz continuous gradient and is sufficiently regular,  cf. \cite{sanzserna2021,paulin2024} . Then, Assumption~\ref{ass} can be verified and we can show that there exits a constant $\tilde{\mathbf{C}}>0$ depending on the dimension $d$, the second moment of $\nu$ and the strongly convexity constant and the Lipschitz constant of $U$ such that for all $l\in\mathbb{N}$,
\begin{align*}
        \mathcal{W}_2(\nu \pi_{\delta}^l, \nu p_{\delta l})\le \tilde{\mathbf{C}}\delta^2\, .
\end{align*}
Alternatively, relaxing the strong convexity condition on $U$ and assuming the assumptions on $U$ of \cite[Theorem 10]{schuhwhalley2024} it holds uniform in time error bounds for the UBU scheme in $L^1$ Wasserstein distance of the form 
    \begin{align*}
        \mathcal{W}_1(p_{l\delta}\nu,\pi_{\delta}^l \nu)\le \tilde{\mathbf{C}} \delta^2\, .
    \end{align*}

\end{example}

\begin{example}[Hamiltonian Monte Carlo (HMC)]
    The Theorem~\ref{thm:uit_num} can also be carried over to time-discrete Markov chains $(\mathbf{X}_k)_{k\in\mathbb{N}}$, where the transition kernel $\pi$ is not exactly implementable but an approximation has to be considered, as for instance for the Hamiltonian Monte Carlo algorithm \cite{neal2011}. 
    A typical approximation scheme of the exact Hamiltonian dynamics is given by the velocity Verlet integration or Leapfrog method. Denote by $(\mathbf{X}_k^{\delta})_{k\in\mathbb{N}}$ the Markov chain of the unadjusted HMC and by $\pi^{\delta}$ the corresponding transition kernel.
    If $U$ is convex and gradient Lipschitz, contractivity can be shown for the exact dynamics in $L^2$ Wasserstein distance, see \cite[Lemma 6]{chen2022}, and a local error bound between the approximation and the exact chain in $L^2$ Wasserstein distance (similarly to the proof of \cite[Theorem 8]{bourabee2023a}). Further, under the above conditions it is straightforward to prove a uniform second moment bound.

    Then analogously to \Cref{thm:uit_num} and \Cref{cor:disc_strongerror}, it holds for any initial measure $\nu$ with finite second moment and $l\in \mathbb{N}$
    \begin{align*}
        \mathcal{W}_2(\nu \pi^l, \nu \pi_{\delta}^l)\le \mathbb{E}[|\mathbf{X}_l-\mathbf{X}_l^{\delta}|^2|\{\mathbf{X}_0=\mathbf{X}^{\delta}_0\sim \nu\}]^{1/2} \le \tilde{\mathbf{C}} \delta\, ,
    \end{align*}
    where $\tilde{\mathbf{C}}\in[0,\infty)$ depends on the second moment of $\nu$ and properties of $U$.
\end{example}

\section{Mean-field particle systems and its limiting process}\label{sec:meanfield}
\subsection{Convergence result}

Given $N\in \mathbb{N}$, we consider a mean-field particle system with $N$ particles which we denote by $\mathcal{X}^N=(X_t^{i,N})_{t\ge 0, i=1,\ldots, N}$. Denote by $p_t$ the corresponding transition function for the $N$-particle system at time $t\ge 0$.
Let each component of the particle system converge to the limiting process $(\bar{X}_t)_{t\ge 0}$ in an appropriate sense as $N\to \infty$.
A typical example forms the \textit{McKean-Vlasov process} given by a system of SDEs
\begin{align*}
    \rmd X^{i,N}_t = b(X^{i,N}_t, \mu_{\mathcal{X}^N_t}) \rmd t + \sigma(X^{i,N}_t , \mu_{\mathcal{X}^N_t})\rmd B^i_t\, ,
\end{align*}
where $b:\mathbb{R}^d\times \mathcal{P}(\mathbb{R}^d)\to \mathbb{R}^d$, $\sigma:\mathbb{R}^d\times \mathcal{P}(\mathbb{R}^d)\to \mathbb{R}^{d\times d}$, $\mu_{\mathcal{X}^N_t}=\frac{1}{N}\sum_{i=1}^N \delta_{X_t^{i,N}}$ and $(B_t^{i})_{t\ge 0}$ are $N$ i.i.d. standard Brownian motions.
Its corresponding limit process is given by the \textit{nonlinear McKean-Vlasov process}
\begin{align*}
    \rmd \bar{X}_t = b(\bar{X}_t, \bar{\mu}_t)\rmd t + \sigma (\bar{X}_t, \bar{\mu}_t )\rmd B_t\, ,
\end{align*}
where $\bar{\mu}_t=\Law(\bar{X}_t)$.

In the following we denote by $\bar{X}^N=(\bar{X}_t^{i,N})_{t \ge 0, i=1, \ldots, N}$ $N$ i.i.d. copies of the limit process and by $\bar{p}_t$ the transition function of one component of the limit process. 

Similarly to the two previous frameworks, we assume contractivity for the mean-field particle process, a local-in-time bound between the particle system and the limit process and a uniform in time bound for the limit process: 

\begin{assumption}\label{ass_meanfield} We impose:
\begin{enumerate}
\item \label{ass1_m} (Contractivity for the mean-field particle system) There exists $\lambda>0$ and $C<\infty$ such that for all probability measures $\nu, \mu$ on $\mathbb{R}^d$ and $t>0$ it holds
\begin{align*}
\d(\nu^{\otimes N} p_t^N, \eta^{\otimes N} p_t^N)\le C e^{-\lambda t} \d(\nu^{\otimes N}, \eta^{\otimes N})\, .
\end{align*}

\item \label{ass2_m}  (Finite time propagation of chaos) There exists a time $\tau>0$, a constant $\alpha>0$ and a function $\mathbf{M}:\mathcal{P}(\mathbb{R}^{d})\to [0,\infty)$ such that for all probability measures $\nu$ on $\mathbb{R}^d$ and all $N\in\mathbb{N}$ it holds
\begin{align*}
\sup_{t\le \tau} \d (\nu^{\otimes N} p_{t}^N, (\nu \bar{p}_{t})^{\otimes N})\le N^{-\alpha} \mathbf{M}(\nu).
\end{align*}

\item \label{ass3_m}  (Uniform control for the limit process) For all probability measures $\nu$ on $\mathbb{R}^d$, there exists a constant $\mathbf{C}(\nu)< \infty$ such that 
\begin{align*}
\sup_{t\ge 0} \mathbf{M}(\nu\bar{p}_t) \le \mathbf{C}(\nu)\, .
\end{align*}

\end{enumerate}
\end{assumption}

\begin{theorem}\label{thm:mf_meas} Suppose \Cref{ass_meanfield} holds. Let $\nu$ be a probability measure on $\mathbb{R}^d$. Then, there exists a constant $\tilde{\mathbf{C}}>0$ such that for all $t\ge 0$ and $N\in \mathbb{N}$
\begin{align*}
    \d ((\nu \bar{p}_t)^{\otimes N}, \nu^{\otimes N} p_t^N) \le \tilde{\mathbf{C}} N^{-\alpha}\, .
\end{align*}
Moreover, the constant $\tilde{\mathbf{C}}$ depends on $\lambda$, $C$ and $\mathbf{C}(\nu)$ given in \Cref{ass_meanfield} and is independent of $N\in\mathbb{N}$.
\end{theorem}

\begin{proof}
    Let $t=k\tau+m$ for some $k\in \mathbb{N}$ and $m\in[0,\tau)$. 
    Then,
    \begin{align*}
        &\d ((\nu \bar{p}_t)^{\otimes N}, \nu^{\otimes N} p_t^N)
        \\ & \le \sum_{i=0}^{k-1} \d ((\nu \bar{p}_{\tau (k-i)})^{\otimes N}p_{\tau i+m}^N, (\nu \bar{p}_{\tau (k-i-1)})^{\otimes N} p_{\tau(i+1)+m}^N)+ \d ((\nu \bar{p}_{\tau k})^{\otimes N} p_{m}^N,(\nu \bar{p}_{\tau k+m})^{\otimes N})
        \\ & =\sum_{i=0}^{k-1} \d ((\nu \bar{p}_{\tau (k-i)})^{\otimes N}p_{\tau i+m}^N, (\nu \bar{p}_{\tau (k-i-1)})^{\otimes N} p_{\tau}^N  p_{\tau i+m}^N)+\d ((\nu \bar{p}_{\tau k})^{\otimes N} p_{m}^N,(\nu \bar{p}_{\tau k+m})^{\otimes N})
        \\ & \le \sum_{i=0}^{k-1} C e^{-\lambda (i \tau+m)} \d ((\nu \bar{p}_{\tau (k-i)})^{\otimes N}, (\nu \bar{p}_{\tau (k-i-1)})^{\otimes N} p_{\tau}^N )+\d ((\nu \bar{p}_{\tau k})^{\otimes N} p_{m}^N,(\nu \bar{p}_{\tau k+m})^{\otimes N})
        \\ & \le  \sum_{i=0}^{k-1} C e^{-\lambda (i \tau+m)} N^{-\alpha} \mathbf{M}(\nu \bar{p}_{k-i-1})+ N^{-\alpha} \mathbf{M}(\nu \bar{p}_k)
        \\&  \le  \left[\sum_{i=0}^{k-1} C e^{-\lambda (i \tau+m)}+1\right] N^{-\alpha} \mathbf{C}(\nu)
        \\ & \le N^{-\alpha} \mathbf{C}(\nu)C\left[1+\frac{e^{-\lambda m}}{1-e^{-\lambda \tau}}\right]\le N^{-\alpha} \mathbf{C}(\nu)C \frac{2}{1-e^{-\lambda \tau}}\, ,
    \end{align*}
    where we used the triangle inequality and applied iteratively \Cref{ass_meanfield}\eqref{ass1_m}, \Cref{ass_meanfield}\eqref{ass2_m} and \Cref{ass_meanfield}\eqref{ass3_m}.
\end{proof}

\begin{remark}
    Similarly to  \cite[Proposition 2.2]{angeli2023uniform} and \Cref{thm:uit_avg_semi}, we can show uniform in time propagation of chaos bounds for the semigroups of the mean-field particle system to the semigroup of the limiting process. Denote by $\cP_t^N$ the semigroup associated with the mean-field particle process $(X_t^{i,N})_{t\ge 0, i=1, \ldots, N}$, and by $\bar{\cP}_t$ the semigroup associated to the limiting process  $(\bar{X}_t)_{t\ge 0}$. 
    We assume strong exponential stability for $\bar{\cP}_t$, i.e., there exists $\lambda>0$ and $C_0>0$ such that for all $f\in \mathcal{C}_b^2(\mathbb{R}^n)$ with $\bar{\mu}(f)=0$, $\| \bar{\cP}_t f\|_{\mathcal{C}_b^2} e ^{-\lambda t}$, 
    and finite time propagation and a uniform control, i.e., there exist positive functions $\phi, \Phi:\mathbb{R}^{Nd}\to [0, \infty)$ and constants $C_1>0$, $\alpha>0$ and $\tau >0$ such that for all $N >0$ and $\mathbf{x}_N=(x_1,\ldots, x_N)\in \mathbb{R}^{Nd}$ it holds $|\mathbb{E}_{\mathbf{x}_N}[f(X_\tau^{1,N})-f(\bar{X}_{\tau})]|\le C_1 \| f\|_{\mathcal{C}_b^2}\phi(\mathbf{x}_N)N^{-\alpha}$
    and $\sup_{t\ge 0} |\mathbb{E}_{\mathbf{x}_N}[\phi((X_t^{i,N})_{i=1}^N]|\le \Phi(\mathbf{x}_N)$.    
    Then it follows analogously to  \Cref{thm:uit_avg_semi} that for all $N\in \mathbb{N}$, $\mathbf{x}_N\in \mathbb{R}^{Nd}$
    \begin{align*}
        \sup_{t\ge 0}|\mathbb{E}_{\mathbf{x}_N}[f(X_t^{N,1})-f(\bar{X}_t)]|\le \frac{C(\|f\|_{\mathcal{C}_b^2}+\bar{\mu}(f))\Phi(\mathbf{x}_N)}{1-e^{-\lambda \tau}}\, .
    \end{align*}
\end{remark}

\subsection{Applications}

Consider the McKean-Vlasov diffusion given by the following system of SDEs: Fix $N\in \mathbb{N}$. 
\begin{align} \label{eq:meanf}
    \rmd X_t^{i,N}=\left[- \nabla U(X_t^{i,N}) +\frac{1}{N}\sum_{j=1}^N \nabla W(X_t^{i,N}-X_t^{j,N})\right] \rmd t + \sqrt{2}\rmd B_t^{i,N}\, ,
\end{align}
where $U,W:\mathbb{R}^d\to \mathbb{R}$ are the confining and the interaction potential, respectively. 
The corresponding nonlinear diffusion is given by 
\begin{align} \label{eq:nonl}
    \rmd \bar{X}_t=\left[- \nabla U(\bar{X}_t) +\int_{\mathbb{R}^d} \nabla W(\bar{X}_t-z)\rmd \bar{\mu}_t(z)\right] \rmd t + \sqrt{2}\rmd B_t
\end{align}
with $\bar{\mu}_t=\Law(\bar{X}_t)$.
Assume $U$ is strongly convex, $W$ is symmetric and convex and $\nabla W$ is locally Lipschitz and has polynomial grow of order $p\ge 1$. Then contractivity for the mean-field particle system holds in $L^2$ Wasserstein distance. If further the initial distribution has finite second moments, then uniform second moment bounds can be proven for the nonlinear diffusion. Further, by \cite{mckean1967} one can show finite in time propagation, which, together with the two previous observations, leads to uniform in time propagation of chaos by \Cref{thm:mf_meas}. In particular for any probability measure $\nu\in\mathcal{P}^2(\mathbb{R}^d)$ the processes $(X_t^{i,N})_{t\ge 0, i=1, \ldots N}$ given by \eqref{eq:meanf} and $(\bar{X}_t^i)_{t\ge 0,i=1, \ldots,N}$ being $N$ i.i.d copies of solutions to \eqref{eq:nonl} with initial conditions $\bar{X}_0^i=X_0^{i,N}\sim \nu$ for all $i=1, \ldots N$ satisfy
\begin{align*}
    \sup_{t\ge 0}\mathbb{E}[N^{-1}\sum_{i=1}^N |\bar{X}_t^i-X_t^{i,N}|^2]\le \tilde{\mathbf{C}}/N\,.
\end{align*}
The constant $\tilde{\mathbf{C}}\in[0,\infty)$ depends only on the second moment of the initial distribution $\nu$ and on properties of $V$ and $W$ and is independent of $N$.
Alternatively, it is possible to show directly uniform in time propagation of chaos bounds, see \cite{malrieu2001}.

\begin{example}[Further examples]
    We note that under relaxed conditions on $U$ and $W$, i.e., $U$ strongly convex at infinity and $W$ Lipschitz continuous, it is possible to verify first Assumption~\ref{ass_meanfield} in $L^1$ Wasserstein distance and conclude that uniform in time propagation of chaos of the form $\sup_{t\ge 0}\mathbb{E}[N^{-1}\sum_{i=1}^N |\bar{X}_t^i-X_t^{i,N}|]\le \tilde{\mathbf{C}}\sqrt{N}$ holds. Alternatively, one can show directly uniform in time propagation of chaos, see \cite{eberle2020}.

    Under similar assumptions uniform in time propagation of chaos can be shown for second order McKean-Vlasov diffusion, see e.g. \cite{schuh2024}.
\end{example}
\appendix

\section{Warning: Time-inhomogeneous setting}\label{sec:appTIH}
The following is an example to show that one must be careful when generalising the results of this paper to other approximations. There is certain structure in the types of approximation we have studied here. In particular, an implicit time homogeneity is assumed. When, instead, there exists an interplay between the approximation index, say $\delta$, and time $t$, even fairly simple and nicely behaved systems may exhibit convergence which is not uniform in time.

\begin{example}
Suppose we have a system
    \begin{equation}\label{eqn:true}
        \rmd X_t = -X_t \rmd t + \rmd W_t
    \end{equation}
    and a corresponding `approximation' given by
    \begin{equation}\label{eqn:approx}
        \rmd X^\delta_t = \left[-X^\delta_t + \mathbb{I}_{t \in \left[\frac{1}{\delta}, \frac{1}{\delta} +1\right]}\right]\rmd t + \rmd W_t \,,
    \end{equation}
    where the same Brownian motion $(W_t)_{t\ge 0}$ is considered in both equations, for simplicity.

    Indeed, we have that
    \begin{equation}
        \rmd (X^\delta_t - X_t) =\left[ -(X^\delta_t - X_t) + \mathbb{I}_{t \in \left[\frac{1}{\delta}, \frac{1}{\delta} +1\right]}\right]\rmd t, 
    \end{equation} which has the following solution
        \begin{equation}
        X^\delta_t - X_t = \begin{cases}
                  0 & t\leq \frac{1}{\delta} \\
      t-\frac{1}{\delta} & \frac{1}{\delta}\leq t\leq \frac{1}{\delta}  + 1\\
      e^{-(t-(\frac{1}{\delta} + 1))} &  x \geq \frac{1}{\delta} + 1 \, . 
        \end{cases}, 
    \end{equation}
    
    Clearly, for every $\delta >0$, there is exponential convergence to equilibrium. We also have finite time convergence of $X_t^\delta$ to $X_t$ (for every given time interval if $\delta$ is chosen sufficiently small, the dynamics in \eqref{eqn:approx} is the same as that of \eqref{eqn:true} during that time interval). Finally, the equilibrium these two processes go to is the same. However, there certainly is not uniform in time convergence. See this by noticing that for all $\delta >0$,
    \begin{equation}
        X^\delta_{\frac{1}{\delta}+1} - X_{\frac{1}{\delta}+1} = 1\, ,
    \end{equation}
    which does not converge to $0$ as $\delta \rightarrow 0$.

    This makes the case that, for the general result, one needs the convergence to equilibrium to be uniform in the approximation parameter, otherwise one can get that the error `runs away' in time.
\end{example}
While the above example is rather particular, it is reasonable to imagine that for an approximation, taking the, say, discretisation parameter small may not reduce the error made once it comes, but simply ensure that the error doesn't come until a later point. In this case one cannot obtain uniform in time convergence.

\subsection*{Acknowledgements}
We would like to thank Michela Ottobre for bringing this topic to KS's attention during the Oberwolfach Seminar on `Constrained Dynamics, Stochastic Numerical Methods and the Modeling of Complex Systems' (May 2024) and for her valuable comments and feedback during the development of the work. 

\bibliography{references}
\bibliographystyle{ieeetr}

\newpage

\end{document}